\documentclass[12pt]{article}
\usepackage{mathrsfs}
\usepackage{amsmath,amsfonts,amssymb,rotating,amsthm}
\usepackage{hyperref}
\usepackage{array}
\usepackage{color}
\def\qed{\nopagebreak\hfill{\rule{4pt}{7pt}}}
\def\proof{\noindent {\it{Proof.} \hskip 2pt}}

\newtheorem{theo}{Theorem}[section]

\newtheorem{lemm}[theo]{Lemma}
\newtheorem{prop}[theo]{Proposition}
\newtheorem{coro}[theo]{Corollary}
\newtheorem{conj}[theo]{Conjecture}

\newtheorem{remark}[theo]{Remark}

\newdimen\Squaresize \Squaresize=11pt
\newdimen\Thickness \Thickness=0.7pt
\def\Square#1{\hbox{\vrule width \Thickness
   \vbox to \Squaresize{\hrule height \Thickness\vss
    \hbox to \Squaresize{\hss#1\hss}
   \vss\hrule height\Thickness}
\unskip\vrule width \Thickness} \kern-\Thickness}

\def\Vsquare#1{\vbox{\Square{$#1$}}\kern-\Thickness}

\def\moins{\raise 1pt\hbox{{$\scriptstyle -$}}}

\begin{document}

\begin{center}
{\large \bf Some Ratio Monotonic Properties of a New Kind of Numbers introduced by Z.-W. Sun}
\end{center}

\begin{center}
Brian Y. Sun
\\[8pt]
Center for Combinatorics, LPMC-TJKLC\\
Nankai University, Tianjin 300071, P. R. China\\[6pt]

Email: 
{\tt brian@mail.nankai.edu.cn}\\
\end{center}

\vspace{0.3cm} \noindent{\bf Abstract.} Recently, Z. W. Sun introduced a new kind of numbers $S_n$ and also posed a conjecture on ratio monotonicity of combinatorial sequences related to $S_n$.  In this paper, by investigating some arithmetic properties of $S_n$, we give an affirmative answer to his conjecture. Our methods are based on a newly established criterion and interlacing method for log-convexity, and also the criterion for ratio log-concavity of a sequence due to Chen, Guo and Wang.


\noindent {\bf AMS Classification 2010:} 05A20; 05A10; 11B65; 11B37

\section{Introduction}\label{Sec-1:introduction}
Recently, Z. W. Sun \cite{sun1,sun2} put forward  a series of conjectures about monotonicity of sequences of forms $\{z_{n+1}/z_n\}_{n\geq 0}^\infty$, $\{\sqrt[n]{z_n}\}_{n\geq 1}$ and $\{\sqrt[n+1]{z_{n+1}}/\sqrt[n]{z_{n}}\}_{n\geq 1}$, where $\{z_n\}_{n\geq 0}$ is number-theoretic or combinatorial sequences of positive numbers. By ratio monotonicity of a sequence $\{z_n\}_{n\geq 0}$ of positive numbers, we mean the sequence $\{z_{n}/z_{n-1}\}_{n\geq 1}$ is a monotonic sequence.
Up to now, much valuable progress has been made in this field by many scholars, including Chen et al. \cite{cgw}, Hou et al. \cite{hsw}, Luca and St\u{a}nic\u{a} \cite{LS}, Wang an Zhu \cite{wz} and Zhao \cite{zhao}, etc. The main object of this paper is to prove a conjecture due to Z. W. Sun \cite{sun1} on ratio monotonicity of a new kind of number.
This new kind of number is also introduced by him in \cite{sun1}.
It is given as follows:
\begin{equation}\label{Seq-S}
S_n=\sum_{k=0}^n\binom{n}{k}^2
\binom{2k}{k}(2k+1),\,\,n=0,1,2,\ldots.
\end{equation}

 By applying the Zeilberger's algorithm \cite{a=b,zeil} developed in \cite{chm,holo,fast}, we can not find a three-term recurrences for $S_n$ . But they all can give us the following four-term recurrences, i.e.,
 \begin{equation}\label{Rec:Seq-S}
 \begin{split}
 &9(n+1)^2S_n-(19n^2+74n+87)S_{n+1}+(n+3)(11n+29)S_{n+2}\\
 &-(n+3)^2S_{n+3}=0.
 \end{split}
 \end{equation}

 Note that all progress and results (mentioned above) related to this subject up to now can only be used to tackle with number theoretic or combinatorial sequences satisfying special expressions or three-term recurrences. That's why the following  conjecture is not solved yet.

\begin{conj}\emph{(\cite[Conjecture 5.2(ii)]{sun1})}\label{Conj:conj-1}
 The sequence $\{S_{n+1}/S_n\}_{n\geq 3}$ is strictly increasing to the limit $9$, and the sequence $\{\sqrt[n+1]{S_{n+1}}/\sqrt[n]{S_n}\}_{n\geq 1}$ is strictly decreasing to the limit $1$.
\end{conj}
Recall that ratio monotonicity is closely related to log-convexity and log-concavity.
Let $\{z_n\}_{n\geq 0}^\infty$ be a sequence of positive numbers. It is called \emph{log-convex(resp.log-concave)} if for all $n\geq 1$
\begin{equation}
\label{Ineq:log-cv and log-cc}
z_{n-1}z_{n+1}\geq z_n^2~~~(\text{resp.}z_{n-1}z_{n+1}\leq z_n^2).
\end{equation}
 Meanwhile, the sequence $\{z_n\}_{n\geq 0}^\infty$ is called strictly log-convex(resp. log-concave) if the inequality in \eqref{Ineq:log-cv and log-cc} is strict for all $n\geq 1$. Many criterions have been established to show that a number of number-theoretic and combinatorial sequences of positive numbers are log-convex or log-concave, see, for example, Brenti \cite{brenti}, Chen \cite{chen}, Chen and Xia \cite{cx}, Chen et al. \cite{cgw}, Do\v{s}li\'{c} \cite{D}, Do\v{s}li\'{c} and Veljan \cite{DV}, Liu and Wang \cite{WL}, Sun and Zhao \cite{sz}, Wang and Yeh \cite{wy}, Yang and Zhao \cite{yz} and Stanley \cite{stanley}. In terms of the inequality in \eqref{Ineq:log-cv and log-cc}, it is clear that a sequence $\{z_n\}_{n\geq 0}^{\infty}$ is log-convex(resp. log-concave) if and only if the sequence $\{\frac{z_{n+1}}{z_n}\}_{n\geq 0}$ is increasing(resp. decreasing). The log-convex and log-concave sequences have been extensively investigated as they are often arise in combinatorics, algebra, geometry, analysis, probability and statistics, the reader can refer to \cite{brenti,wy,stanley} for details.

 In this paper, our aim is to study some properties of the sequence $\{S_n\}_{n=0}^\infty$ and then confirm the Conjecture \ref{Conj:conj-1}, i.e.,
  \begin{theo}\label{Thm:confirm conj}
  Conjecture \ref{Conj:conj-1} is true.
  \end{theo}
  We prove Theorem \ref{Thm:confirm conj} through the following steps. Firstly, we give a three-term recurrence for $S_n$, the technique is based on Guo and Liu \cite{GL}. And then prove the first part of Conjecture \ref{Conj:conj-1} by using a newly established criterion for log-convexity and also employing the interlacing method \cite{D,WL} developed for log-convexity and log-concavity. In the end, we prove the second part of Conjecture \ref{Conj:conj-1} with the criterion for ratio log-convexity of a sequence, which was built by Chen, Guo and Wang \cite{cgw}.
 \section{A new criterion for log-convexity}
 \label{Sec-3: new criterion for lcc}
 In this section, we will establish a criterion for log-convexity of a sequence satisfying three-term recurrence relationship. As a matter of fact, Yang and Zhao \cite{yz} gave a similar criterion for log-concavity of a sequence satisfying three-term recurrence. Essentially, the following criterion for log-convexity can be considered as a counterpart of the criterion for log-concavity due to Yang and Zhao \cite{yz}.
\begin{theo}\label{Thm:new criterion}
Let $\{z_n\}_{n\geq 0}$ be a positive sequence satisfying the following three-term recurrence
relation:
\begin{equation}\label{Rec:three-term}
a(n)z_{n+1}+b(n)z_n+c(n)z_{n-1}=0, \,\,n\geq 1,
\end{equation}
where $a(n), b(n),c(n)$  are real functions in $n$.
 Suppose that there exists an integer $N$  such that for any $n>N$,
 \begin{itemize}
 \item[(i)] it holds $a(n)>0,$ and
 \item[(ii)]$\Delta_n=b(n)^2-4a(n)c(n)\geq 0,$ and
 \item[(iii)]$\frac{-b(n)-\sqrt{\Delta_n}}{2a(n)}\leq \frac{z_n}{z_{n-1}}\leq \frac{-b(n)+\sqrt{\Delta_n}}{2a(n)}.$
 \end{itemize}
 Then the sequence $\{z_n\}_{n\geq N}$ is log-convex, namely, $z_n^2\leq z_{n-1}z_{n+1}.$
\end{theo}
\begin{proof} Let $q_n=\frac{z_n}{z_{n-1}}$.
It suffices to show that $q_n\leq q_{n+1}$ for any $n\geq N.$ On one hand, the conditions $(i)$ and $(ii)$ imply that
$$a(n)q_n^2+b(n)q_n+c(n)\leq 0,\,\mbox{for} \,n\geq N.$$
Since $\{z_n\}_{n\geq 0}$ is a positive sequence, so is $\{q_n\}_{n\geq 1}.$
Thus, the above inequality is equivalent to  the following
\begin{equation}\label{equu-1}
a(n)q_n+b(n)+\frac{c(n)}{q_n}\leq 0, \,\, for \,\,n\geq N.
\end{equation}
On the other hand, dividing both sides of \eqref{Rec:three-term} by $z_n$, we obtain
\begin{equation}\label{equu-2}
a(n)q_{n+1}+b(n)+\frac{c(n)}{q_n}=0.
\end{equation}
Combining \eqref{equu-1} and \eqref{equu-2}, we get
$$a(n)q_{n+1}\geq a(n)q_{n},\,\,\mbox{for}\,\,n\geq N.$$
By the condition $(i)$, we have  $q_{n+1}\geq q_{n}$ for any $n\geq N$.
This completes the proof.
\qed
\end{proof}

 \section{The interlacing method}\label{Sec-2:the interlacing method}
 The interlacing method can be found in \cite{WL}, yet it was formally considered as a method to solve logarithmic behavior of combinatorial sequences by Do\u{s}li\'{c} and Veljan \cite{DV}, in which is also called sandwich method.

 To be self-contained in our paper. Let us give a simple introduction to the method.
 Suppose that $\{z_n\}_{n\geq 0}^\infty$ is a sequence of positive numbers. We define the sequence of consecutive quotients, i.e.,
 $$q_n=\frac{z_n}{z_{n-1}},\,\,\,n\geq 1.$$

 By the inequality in \eqref{Ineq:log-cv and log-cc}, the log-convexity or log-concavity of a sequence $\{z_n\}_{n\geq 0}$ is equivalent, respectively, to $q_n\leq q_{n+1}$ or $q_n\geq q_{n+1}$ for all $n\geq 1$. So it suffices to consider whether the sequence $\{q_n\}_{n\geq 1}$ decreases or increases, i.e., the ratio monotonicity.
 To prove $\{q_n\}_{n\geq 1}$ increases(resp. decreases), it is enough to
 find an increasing(resp. a decreasing) sequence $\{b_n\}_{n\geq 0}$ such that
 \begin{equation}\label{Ineq: interlacing mehtod}
 b_{n-1}\leq q_n\leq b_n\,\,(\text{resp.~}b_{n-1}\geq q_n\geq b_n)
 \end{equation}
  holds for all $n\geq 1$, or at least for all $n\geq N$ for some positive integer $N$. Clearly, this implies $q_n\leq q_{n+1}$(resp. $q_n\geq q_{n+1}$ ) since we have $\ldots\leq b_{n-1}\leq q_{n}\leq b_n\leq q_{n+1}\leq b_{n+1}\leq \ldots$(resp. $\ldots\geq b_{n-1}\geq q_{n}\geq b_n\geq q_{n+1}\geq b_{n+1}\geq \ldots$ ) by
  \eqref{Ineq: interlacing mehtod}. As a summary, we have the following proposition.
  \begin{prop}
  \label{Prop:interlacing method}
  Suppose that $\{z_n\}_{n\geq 0}$ is a sequence of positive numbers. Then for some positive integer $N$, the sequence $\{z_n\}_{n\geq N}$ is log-convex(resp. log-concave) if there exists an increasing(resp. a decreasing) sequence $\{b_n\}_{n\geq 0}$ such that
  \begin{equation*}
 b_{n-1}\leq q_n\leq b_n\,\,(\text{resp.~}b_{n-1}\geq q_n\geq b_n)
 \end{equation*}
 holds for $n\geq N+1$.
  \end{prop}

 \section{Log-behaviors of the sequence related to $S_n$}\label{Sec-4: S is lcc}
 In this section, we will give our main results. We first give a three-term recurrence relationship of $S_n$ in Subsection
 \ref{Sec-4.1: three-term rec for S}.
 And then prove that the sequence $\{S_n\}_{n\geq 0}$ is log-convex by using interlacing method and the criterion built in Section \ref{Sec-3: new criterion for lcc}.
 \subsection{A three-term recurrence for $S_n$}
 \label{Sec-4.1: three-term rec for S}
Usually, when considering the logarithmic behavior or ratio monotonicity of combinatorial sequences, it is difficult to tackle with the sequences satisfying four-term recurrence relationship. So, for a sequence satisfying recurrence of order $3$ or more, it is natural to find whether it has a three-term recurrence or not. In some cases, one can find a three-term recurrence by solving a linear system of equations. However, this method is not a valid way for general case. Fortunately, we can find a three-term recurrence for $S_n$ with this method but we will not intend to use this awkward way here.

In this paper we can deduce a three-term recurrence relationship for $S_n$ indirectly from Guo and Liu \cite{GL}. When attacking some conjectures of Z.W. Sun on the divisibility of certain double-sums of $S_n$,
they first established
$$4nS_n=(n+1)^2f_n-n^2f_{n-1},$$ where
$$f_n=\sum_{k=0}^n\frac{1}{k+1}\binom{2k}{k}
\left(6k{\binom{n}{k}}^2+\binom{n}{k}\binom{n}{k+1}\right).$$
Let $u_n=4nS_n$ and $v_n=(n+1)^2f_n-n^2f_{n-1}$. Then they obtained the following recurrences by applying Zeilberger's algorithm to $u_n$ and $v_n$, respectively.
\begin{equation}\label{Rec: u_n order 3}
\begin{split}
n(n+1)(n+2)(n +3)u_{n+3}-n(n+1)(n+3)(11n+ 29)u_{n+2}\\
+n(n+2)(19n^2+74n+87)u_{n+1}-9(n+ 1)^3(n + 2)u_n = 0,
\end{split}
\end{equation}
and
\begin{equation}\label{Rec: v_n order 3}
\begin{split}
&(n + 2)(n + 3)(128n^4 + 864n^3 + 2016n^2 + 1994n + 693)v_{n+3}-(1408n^6 \\
& + 17696n^5+ 88512n^4 + 225582n^3 + 309049n^2 + 215886n + 59535)v_{n+2}\\
& + (2432n^6+ 30880n^5 + 155712n^4 + 399646n^3 + 550013n^2 + 384657n \\
&+ 106920)v_{n+1}- 9(n + 1)^2
(128n^4 + 1376n^3 + 5376n^2 + 9130n + 5695)v_n \\
&= 0.
\end{split}
\end{equation}
Note that $u_n=v_n$ for all $n\geq 0$, and by combining \eqref{Rec: u_n order 3} and \eqref{Rec: v_n order 3}, they acquired a three-term recurrence relationship for $u_n$ as follows:
\begin{equation}\label{Rec: u_n order 2}
\begin{split}
&n(n+1)(n+2)(4n+3)(4n+7)u_{n+2}-n(4n+3)(4n+11)\\
&(10n^2+30n+23)u_{n+1}+9(n+1)^3(4n+11)(4n+7)u_n=0.
\end{split}
\end{equation}

Obviously, the recurrence \eqref{Rec: u_n order 2}  implies a
three-term recurrence for $S_n$, i.e.,
\begin{equation}\label{three-rec-SS}
\begin{split}
&(n+1)^2(4n-1)(4n+3)S_{n+1}-(4n-1)(4n+7)\left(10n^2+10n+3\right) S_{n}\\
&\quad+9n^2(4n+3)(4n+7)S_{n-1}=0.
\end{split}
\end{equation}

With the above recurrence \eqref{three-rec-SS} in hand, we can prove the log-convexity of $\{S_n\}_{n\geq 0}$, which implies the ratio monotonicity of $\{S_n\}_{n\geq 0}$.
\subsection{The sequence $\{S_n\}_{n\geq 0}$ is log-convex}
\begin{theo}\label{Thm:log-convex-S_n}
 The sequence $\{S_n\}_{n\geq 0}$ is strictly log-convex, i.e., for all $n\geq 1$, there holds $S_n^2<S_{n+1}S_{n-1}$.
\end{theo}
We prove this result with two ways, one is interlacing method introduced in Section \ref{Sec-2:the interlacing method} and the other is the new criterion established in Section \ref{Sec-3: new criterion for lcc}.

\subsubsection{Proof of Theorem \ref{Thm:log-convex-S_n} via Theorem \ref{Thm:new criterion}}
\proof
By the recurrence relation \eqref{three-rec-SS}, we know that
\begin{align*}
a(n)&=(n+1)^2(4n-1)(4n+3),\\
b(n)&=-(4n-1)(4n+7)\left(10n^2+10n+3\right),\\
c(n)&=9n^2(4n+3)(4n+7),\\
\Delta_n&=16384 n^8+81920 n^7+143360 n^6+101888 n^5+19328 n^4-13120 n^3\\
&\quad-6884 n^2-84 n+441\geq 0, \,\,\text{for~}n\geq 0.
\end{align*}
Obviously, $a(n)>0$ for any $n\geq 1.$
Let $s_n=\frac{S_{n}}{S_{n-1}}$, it suffices to show that
\begin{equation}\label{ineq-s_n}
\frac{-b(n)-\sqrt{\Delta_n}}{2a(n)}\leq s_n\leq \frac{-b(n)+\sqrt{\Delta_n}}{2a(n)}.
\end{equation}
Now we first prove the inequality of the right hand side of \eqref{ineq-s_n}.
We proceed our proof by inductive argument. For the sake of brief, let
\begin{align*}
X(n)=\frac{-b(n)-\sqrt{\Delta_n}}{2a(n)},\,\,  Y(n)=\frac{-b(n)+\sqrt{\Delta_n}}{2a(n)}.
\end{align*}
To begin with,
\begin{align*}
&s_1=7,&\\
&X(1)=\frac{\left(759-3 \sqrt{38137}\right)}{168} \thickapprox 1.03059,&\\
&Y(1)=\frac{\left(759+3 \sqrt{38137}\right)}{168} \approx 8.00512.&
\end{align*}
Hence,
\begin{equation}\label{initial-value}
X(1)<s_1<Y(1).
\end{equation}
Suppose that $s_{n}<Y(n)$ for $n\geq 1,$ we need to prove that $s_{n+1}<Y(n+1).$
In terms of inductive hypothesis,
\begin{align*}
s_{n+1}&=\frac{(4n+7)\left(10n^2+10n+3\right)}{(n+1)^2(4n+3)}
-\frac{9n^2(4n+7)}{(n+1)^2(4n-1)s_n}\\
&<\frac{4n+7}{(n+1)^2}\left(\frac{10n^2+10n+3}{4n+3}
-\frac{9n^2}{(4n-1)Y(n)}\right)\\
&=\frac{4n+7}{(n+1)^2}\left(\frac{(4 n-1) (4 n+7) \left(10 n^2+10 n+3\right)+A(n)}{2 (4 n-1) (4 n+3) (4 n+7)}\right),
\end{align*}
where
\begin{align*}
&A(n)\\
&\,=\sqrt{(4 n-1) (4 n+7) \left(1024 n^6+3584 n^5+4032 n^4+1888 n^3+140 n^2-204 n-63\right)}.
\end{align*}
Note that
\begin{align*}
\delta_n&\triangleq Y(n+1)-\frac{4n+7}{(n+1)^2}\left(\frac{(4 n-1) (4 n+7) \left(10 n^2+10 n+3\right)+A(n)}{2 (4 n-1) (4 n+3) (4 n+7)}\right)\\
&=\frac{9(4 n-1) \left(16 n^3+58 n^2+64 n+19\right)-(4n^3+23n^2+44n+28)\sqrt{B(n)}}{2(n+1)^2(n+2)^2(4n-1)(4n+3)(4n+7)}\\
&\quad+\frac{(4n^3+7n^2+2n-1)\sqrt{C(n)}}{2(n+1)^2(n+2)^2(4n-1)(4n+3)(4n+7)},
\end{align*}
where
\begin{align*}
B(n)&=16384 n^8+81920 n^7+143360 n^6+101888 n^5+19328 n^4-13120 n^3\\
&\quad-6884 n^2-84 n+441
\end{align*}
and
\begin{align*}
C(n)&=16384 n^8+212992 n^7+1175552 n^6+3599872 n^5+6693248 n^4\\
 &\quad+7734976 n^3+5418076 n^2+2098212 n+343233.
 \end{align*}

 For $n\geq 1$, since
 \begin{align*}
 &\left(128 n^4+320 n^3+160 n^2-2 n\right)^2-B(n)=4992 n^4+12480 n^3+6888 n^2\\
 &\quad+84 n-441>0
 \end{align*}
 and
 \begin{align*}
 &(128 n^4 + 832 n^3 + 1888 n^2 + 1790 n)^2-C(n)=-150144 n^4-975936 n^3\\
 &\quad -2213976 n^2-2098212 n-343233<0,
 \end{align*}
 we have
 \begin{align*}
 \delta_n&=\frac{9(4 n-1) \left(16 n^3+58 n^2+64 n+19\right)-(4n^3+23n^2+44n+28)\sqrt{B(n)}}{2(n+1)^2(n+2)^2(4n-1)(4n+3)(4n+7)}\\
&\quad+\frac{(4n^3+7n^2+2n-1)\sqrt{C(n)}}{2(n+1)^2(n+2)^2(4n-1)(4n+3)(4n+7)}\\
&>\frac{9(4 n-1) \left(16 n^3+58 n^2+64 n+19\right)}{2(n+1)^2(n+2)^2(4n-1)(4n+3)(4n+7)}\\
&\quad-\frac{(4n^3+23n^2+44n+28)\left(128 n^4+320 n^3+160 n^2-2 n\right)}{2(n+1)^2(n+2)^2(4n-1)(4n+3)(4n+7)}\\
&\quad+\frac{(4n^3+7n^2+2n-1)(128 n^4 + 832 n^3 + 1888 n^2 + 1790 n)}{2(n+1)^2(n+2)^2(4n-1)(4n+3)(4n+7)}\\
&=\frac{1152 n^4+1464 n^3-918 n^2-1626 n-171}{2(n+1)^2(n+2)^2(4n-1)(4n+3)(4n+7)}\\
&>0,\,\,\mbox{for} \,\,n\geq 2.
 \end{align*}
 Thus it follows that $s_{n+1}<Y(n+1)$.

 Now we are to prove the inequality of the left hand side of \eqref{ineq-s_n}.
 Note that
 $$X(n)=\frac{(4 n-1) (4 n+7) \left(10 n^2+10 n+3\right)-\sqrt{B(n)}}{2 (n+1)^2 (4 n-1) (4 n+3)}.$$
 Let
 \begin{equation*}
 L(n)=\frac{(4 n+7) \left(10 n^2+10 n+3\right)}{2 n (n+1)^2 (4 n+3)}.
 \end{equation*}
 Obviously, $L(n)>X(n)$ for all $n\geq 1.$
 Similarly, we can get $s_{n+1}>L(n+1)>X(n+1)$ by induction and the recurrence \eqref{rec-s_n}. We omit the proof here and left it to the interested reader.
 By inductive argument and initial values \eqref{initial-value}, we have \begin{equation}\label{xsy}
 X(n)<s_n<Y(n),\,\, \mbox{for}\,\, \mbox{all}\,\, n\geq 1.
 \end{equation}

 Therefore, by Theorem \ref{Thm:new criterion}, we are done.
 \qed

As a corollary we have
\begin{coro}\label{Cor:S-ratio-limit}
The sequence $\{S_{n+1}/S_n\}_{n\geq 0}$ is strictly increasing, and
$$\lim_{n\rightarrow \infty}\frac{S_{n+1}}{S_n}=9.$$
\end{coro}
\begin{proof}By Theorem \ref{Thm:log-convex-S_n}, it follows that $S_n^2<S_{n+1}S_{n-1}$ for  all $n\geq 1$. Equivalently,
 \begin{equation*}
 \frac{S_n}{S_{n-1}}<\frac{S_{n+1}}{S_n},\,\mbox{for}\,\,n\geq 1
 \end{equation*}
 since $\{S_n\}_{n\geq 0}$ is a sequence of positive real numbers.

 Additionally, if we let
\begin{align*}
s_n=\frac{S_{n}}{S_{n-1}},\,\,s=\lim_{n\rightarrow\infty}s_n.
\end{align*}
Then by dividing $n^4S_n$ on both sides of \eqref{three-rec-SS}, it gives the following recurrence for $s_n,$
\begin{equation}\label{rec-s_n}
s_{n+1}=\frac{(4n+7)\left(10n^2+10n+3\right)}{(n+1)^2(4n+3)}
-\frac{9n^2(4n+7)}{(n+1)^2(4n-1)s_n}.
\end{equation}
 Taking $n\rightarrow\infty$ in \eqref{rec-s_n}, it follows that
\begin{equation*}
s^2-10s+9=0.
\end{equation*}
This gives us $s=1$ and $s=9.$ What's more, since $s_n$ is strictly increasing and
$\lim_{n\rightarrow\infty}Y(n)=9$, it follows that
$$s_1=7<\lim_{n\rightarrow\infty}s_{n+1}=\lim_{n\rightarrow\infty}s_{n}\leq 9$$ by \eqref{xsy}. Therefore, it forces $s\lim_{n\rightarrow\infty}s_{n}= 9$.
This completes the proof.
\qed
\end{proof}

As $\{S_n\}_{n\geq 0}$ is a positive sequence, so we can define a new sequence $\{\sqrt[n]{S_n}\}_{n\geq 1}$. Then we have the following result.
\begin{coro}\label{Cor:S_kaifang-increase}
The sequence $\{\sqrt[n]{S_n}\}_{n\geq 1}$ is strictly increasing.  Moreover,
\begin{equation}\label{Lim:kaifang-limit}
\lim_{n\rightarrow\infty}\sqrt[n]{S_n}=9.
\end{equation}
\end{coro}
\begin{proof}
By Corollary \ref{Cor:S-ratio-limit}, we have
\begin{equation*}
\frac{S_{n+1}}{S_n}>\frac{S_n}{S_{n-1}},\,\,\mbox{for}\,\,n\geq 1.
\end{equation*}
With the fact $S_0=1$, we deduce that
$$S_n=S_0\cdot \frac{S_1}{S_0}\cdot \frac{S_2}{S_1}\cdots\frac{S_n}{S_{n-1}}<\left(\frac{S_{n+1}}{S_n}\right)^n, $$
which implies $$S_n^{n+1}<S_{n+1}^n.$$
This is equivalent to $$(S_n^{n+1})^{\frac{1}{n(n+1)}}<(S_{n+1}^n)^{\frac{1}{n(n+1)}},$$
that is,
$$\sqrt[n]{S_n}<\sqrt[n+1]{S_{n+1}}.$$

Additionally,
 consider that for a real sequence $\{z_n\}_{n=1}^\infty$ with positive numbers, it was shown that
\begin{equation}\label{in-1}
\lim_{n\rightarrow\infty}\inf\frac{z_{n+1}}{z_n}\leq \lim_{n\rightarrow\infty}\inf\sqrt[n]{z_n}, \end{equation}
and
\begin{equation}\label{in-2}
\lim_{n\rightarrow\infty}\sup\sqrt[n]{z_n}\leq \lim_{n\rightarrow\infty}\sup\frac{z_{n+1}}{z_n}, \end{equation}
see Rudin \cite[\S 3.37]{R}. The inequalities in \eqref{in-1} and \eqref{in-2} implies that
$$\lim_{n\rightarrow\infty}\sqrt[n]{z_n}=
\lim_{n\rightarrow\infty}\frac{z_n}{z_{n-1}}$$ if $\lim_{n\rightarrow\infty}\frac{z_n}{z_{n-1}}$ exists. By Corollary \ref{Cor:S-ratio-limit}, it follows \eqref{Lim:kaifang-limit}.

This completes the proof.

\qed
\end{proof}
\subsubsection{Proof of Theorem \ref{Thm:log-convex-S_n} via interlacing method }
Before using the interlacing method, we need a lemma.
\begin{lemm}\label{Lem:h(n-1)-S_n-h(n)}
Let
$$h(n)=9-\frac{9}{2n^2}.$$
Then, for $n\geq 2$, we have $$h(n-1)<s_n<h(n).$$
\end{lemm}
\begin{proof}
By induction, as $h(1)=\frac{9}{2}<s_2=\frac{55}{7}<h(2)=\frac{63}{8}$, so the basic step is valid.
Suppose that $h(n-1)<s_n<h(n)$ for all $n\geq 2$, we proceed to show that $h(n)<s_{n+1}<h(n+1).$

On one hand, note that
\begin{equation}\label{ineq-ratio-lc-S-n-1}
\begin{split}
s_{n+1}&=\frac{(4n+7)\left(10n^2+10n+3\right)}{(n+1)^2(4n+3)}
-\frac{9n^2(4n+7)}{(n+1)^2(4n-1)s_n}\\
&<\frac{4n+7}{(n+1)^2}\left(\frac{10n^2+10n+3}{4n+3}
-\frac{9n^2}{(4n-1)h(n)}\right)\\
&=\frac{4n+7}{(n+1)^2}\left(\frac{72 n^5+54 n^4-36 n^3-36 n^2-2 n+3}{(4 n-1) (4 n+3) \left(2 n^2-1\right)}\right),
\end{split}
\end{equation}
and
\begin{equation}\label{ineq-ratio-lc-S-n-2}
\begin{split}
&\frac{4n+7}{(n+1)^2}\left(\frac{72 n^5+54 n^4-36 n^3-36 n^2-2 n+3}{(4 n-1) (4 n+3) \left(2 n^2-1\right)}\right)-h(n+1)\\
&=\frac{-88 n^2-40 n+15}{2 (n+1)^2 (4 n-1) (4 n+3) \left(2 n^2-1\right)}\\
&<0,\,\,\mbox{for}\,\,n\geq 1.
\end{split}
\end{equation}
Obviously, the above inequality \eqref{ineq-ratio-lc-S-n-1} and \eqref{ineq-ratio-lc-S-n-2} implies $s_{n+1}<h(n+1).$

On the other hand, consider that
\begin{equation}\label{ineq-ratio-lc-S-n-11}
\begin{split}
s_{n+1}&=\frac{(4n+7)\left(10n^2+10n+3\right)}{(n+1)^2(4n+3)}
-\frac{9n^2(4n+7)}{(n+1)^2(4n-1)s_n}\\
&>\frac{4n+7}{(n+1)^2}\left(\frac{10n^2+10n+3}{4n+3}
-\frac{9n^2}{(4n-1)h(n-1)}\right)\\
&=\frac{(4 n+7) \left(72 n^5-90 n^4-72 n^3+10 n^2+14 n-3\right)}{(n+1)^2 (4 n-1) (4 n+3) \left(2 n^2-4 n+1\right)},
\end{split}
\end{equation}
and
\begin{equation}\label{ineq-ratio-lc-S-n-12}
\begin{split}
&\frac{(4 n+7) \left(72 n^5-90 n^4-72 n^3+10 n^2+14 n-3\right)}{(n+1)^2 (4 n-1) (4 n+3) \left(2 n^2-4 n+1\right)}\\
&=\frac{512 n^5-792 n^4-728 n^3+147 n^2+126 n-27}{2 n^2 (n+1)^2 (4 n-1) (4 n+3) \left(2 n^2-4 n+1\right)}\\
&>0,\,\,\mbox{for}\,\,n\geq 1.
\end{split}
\end{equation}
It follows from \eqref{ineq-ratio-lc-S-n-11} and \eqref{ineq-ratio-lc-S-n-12} that
$s_n>h(n-1).$

According to inductive argument, it follows
$$h(n-1)<s_n<h(n),\,\,\mbox{for}\,\, n\geq 2.$$
\qed
\end{proof}
\begin{remark}Corollary \ref{Cor:S-ratio-limit} can be proved easily by invoking this lemma.
\end{remark}
\emph{Proof of Theorem \ref{Thm:log-convex-S_n}.} Clearly, $\{h_n\}_{n\geq 1}$ is an increasing sequence. By Lemma \ref{Lem:h(n-1)-S_n-h(n)}
and Proposition \ref{Prop:interlacing method}, it follows that $\{S_n\}_{n\geq 0}$ is log-convex.
\subsection{The sequences $\{S_n/S_{n-1}\}_{n\geq 1}$ and $\{\sqrt[n]{S_n}\}_{n\geq 1}$ are log-concave}
In this subsection, we will prove the log-concavity of the sequence $\{S_n\}_{n\geq 0}$. The following Theorem \ref{Thm:cgw-ratio-lc} and Theorem \ref{Thm:Chen-Guo-Wang} are indispensable for proving our results.
\begin{theo}[\cite{cgw}]\label{Thm:cgw-ratio-lc}
Let $\{z_n\}_{n\geq 0}$ be the sequence defined by the following recurrence
$$z_n=u(n)z_{n-1}+v(n)z_{n-2}.$$ Assume that $v(n)<0$ for $n\geq 2$.
If there exists a nonnegative integer $N$ and a function $h(n)$ such that for
all $n\geq N+2$,
\begin{itemize}
\item[(i)] $\frac{3u(n)}{4}\leq \frac{z_n}{z_{n-1}}\leq h(n)$;
\item[(ii)]$h(n)^4-u(n)h(n)^3-u(n+1)v(n)h(n)-v(n)v(n+1)<0,$
\end{itemize}
then $\{z_n\}_{n\geq 0}$ is ratio log-concave.
\end{theo}
\begin{theo}[\cite{cgw}]
\label{Thm:Chen-Guo-Wang}
Assume that $k$ is a positive integer. If a sequence $\{z_n\}_{n\geq k}$ is ratio log-concave and
\begin{equation*}
\frac{\sqrt[k+1]{z_{k+1}}}{\sqrt[k]{z_k}}
>\frac{\sqrt[k+2]{z_{k+2}}}{\sqrt[k+1]{z_{k+1}}},
\end{equation*}
then the sequence $\{\sqrt[n]{z_n}\}_{n\geq k}$ is strictly log-concave.
\end{theo}
Now we can give and prove our results.
\begin{theo}\label{S_n-ratio-lc-thm}
The sequence $\{S_n\}_{n\geq 0}$ is ratio log-concave, i.e., the sequence $\{S_n/S_{n-1}\}_{n\geq 1}$ is log-concave.
\end{theo}
\begin{proof}  We prove it by using Theorem \ref{Thm:cgw-ratio-lc}. The recurrence relation \eqref{three-rec-SS} implies that
\begin{equation*}
S_n=\frac{(4 n+3) \left(10 n^2-10 n+3\right)}{n^2 (4 n-1)}S_{n-1}-\frac{9 (n-1)^2 (4 n+3)}{n^2 (4 n-5)}S_{n-1}.
\end{equation*}
To keep notation in Theorem \ref{Thm:cgw-ratio-lc}, here we still let
\begin{align*}
u(n)=\frac{(4 n+3) \left(10 n^2-10 n+3\right)}{n^2 (4 n-1)},\,\,v(n)=\frac{9 (n-1)^2 (4 n+3)}{n^2 (4 n-5)}.
\end{align*}

Consider that
\begin{align*}
\frac{3u(n)}{4}-h(n-1)&=-\frac{3 \left(8 n^5-18 n^4+6 n^3-41 n^2+36 n-9\right)}{4 n^2 (n-1)^2 (4 n-1)}\\
&<0,\,\,\mbox{for}\,\,n\geq 3,
\end{align*}
which implies that
\begin{equation}\label{lower-B-ratio-lc-S-n}
\frac{3u(n)}{4}<h(n-1),\, n\geq 3.
\end{equation}

Additionally,
\begin{equation}\label{ratio-condition-ii}
\begin{split}
&h(n)^4-u(n)h(n)^3-u(n+1)v(n)h(n)-v(n)v(n+1)\\
&=\frac{-D(n)}{16 n^8 (n+1)^2 (4 n-5) (4 n-1)}\\
&<0,\,\,\mbox{for}\,\,n\geq 1,
\end{split}
\end{equation}
where
\begin{equation*}
\begin{split}
D(n)&=331776 n^8-393984 n^7-693360 n^6+524232 n^5+581256 n^4-242028 n^3\\
&\quad-223803 n^2+39366 n+32805.
\end{split}
\end{equation*}
By Lemma \ref{Lem:h(n-1)-S_n-h(n)}, and combining the inequalities \ref{lower-B-ratio-lc-S-n} and \ref{ratio-condition-ii}, it follows
that the sequence $\{S_n\}_{n\geq 0}$ is ratio log-concave by invoking
Theorem \ref{Thm:cgw-ratio-lc}.
\qed
\end{proof}

\begin{theo}\label{kaifang-lc-cor}
The sequence $\{\sqrt[n]{S_n}\}_{n\geq 1}$ is strictly log-concave.
\end{theo}
\begin{proof}
Since
$$\left(\frac{\sqrt{55}}{7}\right)^6-\left(\frac{\sqrt[3]{93}}{\sqrt[6]{5} \sqrt{11}}\right)^6=\frac{89679424}{782954095}>0,$$
it follows that
$$\frac{\sqrt{S_2}}{S_1}=\frac{\sqrt{55}}{7}>
\frac{\sqrt{S_3}}{S_2}=\frac{\sqrt[3]{93}}{\sqrt[6]{5} \sqrt{11}}.$$
By Theorem \ref{S_n-ratio-lc-thm} and \ref{Thm:Chen-Guo-Wang}, we can deduce this result.
\qed
\end{proof}
This implies the following result.
\begin{coro}\label{kaifang-dec-cor}
The sequence $\{\frac{\sqrt[n+1]{S_{n+1}}}{\sqrt[n]{S_n}}\}_{n\geq 1}$ is strictly decreasing.
\end{coro}
\begin{theo}\label{Thm:kaifangsn-limit}
$$\lim_{n\rightarrow \infty}\frac{\sqrt[n+1]{S_{n+1}}}{\sqrt[n]{S_n}}=1.$$
\end{theo}
\begin{proof}
By Corollary \ref{Cor:S_kaifang-increase}, we see that
\begin{equation}\label{kaifang-lim-low-BOUND}
\frac{\sqrt[n+1]{S_{n+1}}}{\sqrt[n]{S_n}}>1.
\end{equation}
In terms of Lemma \ref{Lem:h(n-1)-S_n-h(n)}, we have
$$7\prod_{i=2}^{n}h(i-1)<S_n<7\prod_{i=2}^{n}h(i).$$
Therefore,
\begin{align*}
\log{\left(\frac{\sqrt[n+1]{S_{n+1}}}{\sqrt[n]{S_n}}\right)}&
=\frac{\log{S_{n+1}}}{n+1}-\frac{\log{S_{n}}}{n}\\
&<\frac{\log{\left(7\prod_{i=2}^{n+1}h(i)\right)}}{n+1}-
\frac{\log{\left(7\prod_{i=2}^{n}h(i-1)\right)}}{n},
\end{align*}
and
\begin{align*}
\log{\left(\frac{\sqrt[n+1]{S_{n+1}}}{\sqrt[n]{S_n}}\right)}&
=\frac{\log{S_{n+1}}}{n+1}-\frac{\log{S_{n}}}{n}\\
&>\frac{\log{\left(7\prod_{i=2}^{n+1}h(i-1)\right)}}{n+1}-
\frac{\log{\left(7\prod_{i=2}^{n}h(i)\right)}}{n}.
\end{align*}

By invoking Mathematical software {\tt{Mathematica 10.0}}, we obtain that
\begin{align*}
\lim_{n\rightarrow\infty}\left(\frac{\log{\left(7\prod_{i=2}^{n+1}h(i)\right)}}{n+1}-
\frac{\log{\left(7\prod_{i=2}^{n}h(i-1)\right)}}{n}\right)&=0,\\
\lim_{n\rightarrow\infty}\left(\frac{\log{\left(7\prod_{i=2}^{n+1}h(i-1)\right)}}{n+1}-
\frac{\log{\left(7\prod_{i=2}^{n}h(i)\right)}}{n}\right)&=0.
\end{align*}
Thus,
\begin{align*}
\lim_{n\rightarrow\infty}\log{\left(\frac{\sqrt[n+1]{S_{n+1}}}{\sqrt[n]{S_n}}\right)}=0,
\end{align*}
i.e,
\begin{align*}
\lim_{n\rightarrow\infty}\frac{\sqrt[n+1]{S_{n+1}}}{\sqrt[n]{S_n}}=1.
\end{align*}
This ends the proof.
\qed
\end{proof}
\subsection{Proof of Theorem \ref{Thm:confirm conj}}
\begin{proof}
Corollary \ref{Cor:S-ratio-limit} confirms the former part of Conjecture \ref{Conj:conj-1}.
Additionally, Corollary \ref{Cor:S_kaifang-increase} and
Theorem \ref{Thm:kaifangsn-limit} give an affirmative answer to the latter part of Conjecture \ref{Conj:conj-1}.
Thus, this completes the proof of Theorem \ref{Thm:confirm conj}.
\qed
\end{proof}

\vspace{.3cm}

\noindent{\bf Acknowledgments.} This work was
supported by the National Science Foundation of China.

\end{document}